\documentclass{amsart}%
\usepackage{amsmath}
\usepackage{amsfonts}
\usepackage{amssymb}
\usepackage{graphicx}%
\providecommand{\U}[1]{\protect\rule{.1in}{.1in}}
\newtheorem{theorem}{Theorem}
\theoremstyle{plain}

\newtheorem{corollary}{Corollary}

\newtheorem{example}{Example}

\newtheorem{lemma}{Lemma}

\newtheorem{proposition}{Proposition}
\newtheorem{remark}{Remark}

\numberwithin{equation}{section}
\begin{document}
\title[Multivariate Normal]{ The conditional expectation of the product of the first $n-1$ Hermite
polynomials in a multivariate normal distribution with respect to the $n$-th
variable. A fresh perspective on the Kibble-Slepian formula.}
\author{Pawe\l \ J. Szab\l owski}
\address{Department of Mathematics and Information Sciences, \\
Warsaw University of Technology\\
ul Koszykowa 75, 00-662 Warsaw, Poland }
\email{pawel.szablowski@gmail.com}
\thanks{}
\date{May 2026}
\subjclass[2020]{Primary 33C45, 62H05, ; Secondary 60E07, 60E99}
\keywords{Normal random vector, conditional moments, Hermite polynomials}

\begin{abstract}
We calculate the conditional expectation of $\prod_{j=1}^{n-1}H_{k_{j}}%
(X_{j})$ given $X_{n}=z,$ if random vector $(X_{1},\ldots,X_{n})^{T}$ has
multivariate normal distribution and $H_{n}(x)$ denotes $n$-th Hermite
polynomial. This expectation is a polynomial in $z$ of order $\sum_{j=1}%
^{n-1}k_{j}$. Our formula has an iterative form with respect to $n$. We also
present some auxiliary observations concerning the expansion of the density of
the $n$-dimensional normal distribution in the series of the Hermite
polynomials. Mostly concerning the properties of the coefficients of this
expansion. To perform these calculations, we give a few auxiliary formulas
concerning Hermite polynomials and multivariate normal distributions. We apply
this result to obtain exact, simple forms of these expansions for $n=2$ and
$3$, thus looking at known results from a different perspective.

We believe that the results can be very useful in Probability Theory, in
particular in the Theory of distributions, Statistics, and the Theory of
Stochastic Processes, in particular Gaussian ones,

\end{abstract}
\maketitle

\section{Motivation Preliminaries and Introduction}

The Kibble-Slepian expansion (briefly K-S expansion) formula, deals with
expanding the density of multivariate normal distribution in the infinite
series of Hermite polynomials. We will denote these polynomials by $\left\{
H_{n}(x)\right\}  _{n\geq0}$. In the original papers of first Kibble
\cite{Kibble} and later Slepian \cite{Slepian72}, the coefficients of this
expansion are derived directly from the variance--covariance matrix, more
precisely the formula stems from the properties of the characteristic function
of the multivariate normal distribution. Apart from the original papers of
Kibble and Slepian, there were also other approaches to justify this formula
based on slightly different perspectives. Namely, \cite{Foata81} and
\cite{Louck81} exploit mostly symmetries and base on certain combinatorial
properties of variance--covariance matrices. There were also attempts to
generalize the K-S formula by replacing Hermite polynomials with the so-called
$q-$Hermite ones \cite{SzablKib} and an attempt to move from $1$ to $2$
dimensions \cite{Ism16}. There also exists a generalization of the normal
distribution, with Hermite polynomials replaced by $q-$Hermite polynomials,
stemming from the theory of Markov chains. It was done in \cite{Szab10}. This
Markov case with $q=1$ will be included for completeness and will be explained
in the sequel. This case leads to a very particular form of
variance--covariance matrix (see subsection \ref{Mark}).

Our approach is much simpler and more general and in a sense straightforward.
Namely, we use the fact that the density of a multivariate normal distribution
can be expanded in an infinite series of products of Hermite polynomials.
There appear Hermite polynomials since they are orthogonal with respect to the
marginal densities (which are obviously normal). We start with three general
observations concerning the multivariate normal distribution. Namely, the
property that all conditional moments of order $n$ are polynomials of order
$n$ in the conditioning of random variables, symmetry expressed in the fact
that moments of odd order are equal to zero (Isserlis theorem) and certain
particular properties of Hermite polynomials. Below we will show that the
coefficients of the expansions are equal to the expectations of the products
of Hermite polynomials of different orders. More precisely under the
assumption that the random vector $(X_{1},\ldots,X_{n})$ has multivariate
normal distribution on the values of the quantities $E\prod_{j=1}^{n}H_{k_{j}%
}(X_{j})$. It turns out that these quantities are equal to zero for many
combinations of the indices $k_{1},\ldots,k_{n}.$ That is why the change of
the way of summation is possible. This, on the other hand, leads to the K-S
formula. We explain this in detail below. To find values of
coefficients $E\prod_{j=1}^{n}H_{k_{j}}(X_{j})$ is rather a difficult task. We
will find them recursively, indirectly by integrating the conditional
expectations $E\left(  \prod_{j=1}^{n-1}H_{k_{j}}(X_{j})|X_{n}=z\right)  .$ It
will turn out that these functions are linear combinations of polynomials
$\prod_{j=1}^{n-1}H_{s_{j}}(z)$ satisfying condition $\sum_{j=1}^{n-1}%
s_{j}\leq\sum_{j=1}^{n-1}k_{j}.$ In this way we make a few new interesting
statements about multivariate normal distribution. First one is that we find
the formula for the function $E\left(  \prod_{j=1}^{n-1}H_{k_{j}}(X_{j}%
)|X_{n}=z\right)  $, Second, we show that this function depends on
polynomials $\prod_{j=1}^{n-1}H_{s_{j}}(z)$ and Third, we show that the
coefficients $E\prod_{j=1}^{n}H_{k_{j}}(X_{j})$, as being equal to integrals
of
\[
\int E\left(  \prod_{j=1}^{n-1}H_{k_{j}}(X_{j})|X_{n}=z\right)  H_{k_{n}%
}(z)f_{N}(z)dz,
\]
with $f_{N}(z)$ denoting normal $N(0,1)$ distribution, depend heavily on the
integer functions defined by the integrals%
\[
\int\prod_{j=0}^{n}H_{k_{j}}(z)f_{N}(z)dz.
\]
Despite its simple form little is known about this function especially for
$n>3$. We will find its generating function.

The conditional expectations
\[
E\left(  \prod_{j=1}^{m}H_{k_{j}}(X_{j})|X_{m+1},\ldots,X_{n}\right)  ,
\]
are of interest by themselves. For most $k^{\prime}s$ these polynomials are
unknown. As stated above we find them for $m\allowbreak=\allowbreak n-1.$

For a specific form of variance--covariance matrix of the multivariate
distribution leading to Markov chains, we adapt results known for the
$q-$Hermite polynomials (presented in \cite{SzablAW}) to the normal case
(i.e., for $q\allowbreak=\allowbreak1)$ and find these conditional
expectations also for $m\allowbreak=\allowbreak1.$

As stated above, these polynomials are helpful in finding quantities
$E\prod_{j=1}^{n}H_{k_{j}}(X_{j})$ that are crucial in finding multivariate
expansion of the density of multivariate normal distribution in the product of
marginal $1$-dimensional distribution times an infinite series in product of
Hermite polynomials $\prod_{j=1}^{n}H_{k_{j}}(x_{j})$. This expansion is the
K-S expansion. In this way, we examine the K-S formula from different
perspectives and provide a simple, elementary proof of this expansion. We
illustrate the theory developed here by rigorously analyzing the cases $n=2,3$, and
partially $4$.

The paper is organized as follows. In the rest of this section, more precisely
in Subsection \ref{AUX}, we recall well-known facts about multivariate normal
distributions, mostly concerning finding the marginal distributions and
distributions of affine transformations. We also recall well-known facts about
Hermite polynomials and prove some new ones. The next Section\ref{GmN} is
devoted to some general observations concerning multivariate normal
distribution and presentation of our main results. In particular, we present a
system of linear equations that are satisfied by the conditional moments
$E\left(  \prod_{j=1}^{n-1}H_{k_{j}}(X_{j})|X_{n}=z\right)  $ and we solve
these equations. The solution for some $n$ depends on the values $E\prod
_{j=1}^{m}H_{k_{j}}(X_{j})$ for $m<n.$ Thus, the system of equations can be
solved iteratively.

The last Section \ref{23} is a presentation of elementary and short proofs of
the so-called Mehler formula and K-S formula for $n\allowbreak=\allowbreak3$
and present some remarks concerning the case $n\allowbreak=\allowbreak4$ as
illustration of the theory developed in the paper.

\subsection{Auxiliary results\label{AUX}}

Recall that the standard normal distribution has density%
\[
f_{N}(x)=\frac{1}{\sqrt{2\pi}}\exp\left(  -x^{2}/2\right)  .
\]
The normal distribution $N(m,\sigma)$ with expectation $m$ and standard
deviation $\sigma$ has the following density%
\[
f_{G}(x|m,\sigma^{2})=\frac{1}{\sigma\sqrt{2\pi}}\exp\left(  -\frac{(x-m)^{2}%
}{2\sigma^{2}}\right)  .
\]
Thus, we have $f_{G}(x|m,\sigma^{2})\allowbreak=\allowbreak f_{N}(\frac
{x-m}{\sigma})/\sigma$. Let us also denote
\begin{equation}
f_{CN}(x|y,\rho)\allowbreak=\allowbreak f_{G}(x|\rho y,1-\rho^{2})=\frac
{1}{\sqrt{2\pi(1-\rho^{2})}}\exp\left(  -\frac{(x-\rho y)^{2}}{2(1-\rho^{2}%
)}\right)  \label{cn}%
\end{equation}
and let us call this density conditional normal.

$i$ will always denote imaginary unit, i.e., $i^{2}=-1.$

We will use the usual matrix-vector notation, i.e., all vectors will be
understood to be columns; the upper index $T$ will denote transposition. An
$n\times n$ matrix with $j,k$ entry equal $\sigma_{j,k}$ we will denote
$\left[  \sigma_{j,k}\right]  $ or $\left[  \sigma_{j,k}\right]
_{j,k=1,\ldots n}.$ If we add assumption of symmetry then symmetric $\left[
\sigma_{jk}\right]  $ means naturally that $\sigma_{j,k}\allowbreak
=\allowbreak\sigma_{k,,j}$.

Let us denote by $N\left(  \mathbf{m,\Sigma}\right)  $ the multivariate normal
distribution, i.e., the distribution of the, say, $n-$dimensional random
vector $\mathbf{X}$ with vector of expectations $\mathbf{m\allowbreak
=\allowbreak EX,}$ and variance--covariance matrix $\mathbf{\Sigma
\allowbreak=\allowbreak E}\left(  \mathbf{X}-\mathbf{m}\right)  \left(
\mathbf{X}-\mathbf{m}\right)  ^{T}$. To avoid unnecessary complications we
will always assume that the matrix $\mathbf{\Sigma}$ is positive definite.
Then the distribution $N\left(  \mathbf{m,\Sigma}\right)  $ has the following
density%
\[
f_{mN}(\mathbf{x|m,\Sigma)=}\frac{1}{\sqrt{(2\pi)^{n}\det\mathbf{\Sigma}}}%
\exp\left(  -\left(  \mathbf{x-m}\right)  ^{T}\mathbf{\Sigma}^{-1}\left(
\mathbf{x-m}\right)  /2\right)  ,
\]
where $\mathbf{x}$ is a column with coordinates $x_{1},\ldots,x_{n},$
superscript $T$ denotes transposition. Moreover, we know also, that if an
$n-$dimensional random vector $\mathbf{X}$ has normal distribution
$N(\mathbf{\mu},\mathbf{\Sigma})$ and is split into two blocks $%
\genfrac{[}{]}{0pt}{}{\mathbf{X}_{1}}{\mathbf{X}_{2}}%
$ of dimensions $k$ and $n-k$ and consequently vector $\mathbf{\mu}$ and
matrix $\mathbf{\Sigma}$ are split respectively so that $\mathbf{\mu
\allowbreak=\allowbreak}%
\genfrac{[}{]}{0pt}{}{\mathbf{\mu}_{1}}{\mathbf{\mu}_{2}}%
,$ $\mathbf{\Sigma\allowbreak=\allowbreak}\left[
\begin{array}
[c]{cc}%
\mathbf{\Sigma}_{11} & \mathbf{\Sigma}_{12}\\
\mathbf{\Sigma}_{21} & \mathbf{\Sigma}_{22}%
\end{array}
\right]  $ , with $\mathbf{\Sigma}_{12}\allowbreak=\allowbreak\mathbf{\Sigma
}_{21}^{T},$ of appropriate dimensions, then the marginal distribution of
$\mathbf{X}_{2}$ is normal $N\left(  \mathbf{m}_{2}\mathbf{,\Sigma}%
_{2}\right)  $ and the conditional distribution $\mathbf{X}_{1}|\mathbf{X}%
_{2}=\mathbf{x}_{2}$ is also normal
\begin{equation}
N\left(  \mathbf{\mu}_{1}+\mathbf{\Sigma}_{12}\mathbf{\Sigma}_{22}%
^{-1}(\mathbf{x}_{2}-\mathbf{\mu}_{2}),\mathbf{\Sigma}_{11}-\mathbf{\Sigma
}_{12}\mathbf{\Sigma}_{22}^{-1}\mathbf{\Sigma}_{21}\right)  . \label{margin}%
\end{equation}
Moreover, it is also known that if $\mathbf{Y\allowbreak=\allowbreak}$
$\mathbf{AX+b,}$ where $\mathbf{A}$ is a $m\times n$ matrix of full rank
$m<n,$ then%
\begin{equation}
\mathbf{Y\sim N(Am+b,A\Sigma A}^{T}). \label{afine}%
\end{equation}

The following proposition is an immediate corollary of these observations.

\begin{proposition}
Let $n-$dimensional random vector be normal with $N\left(  \mathbf{0}%
,\mathbf{\Sigma}\right)  $ where the positive definite matrix $\mathbf{\Sigma
}\allowbreak=\allowbreak\left[  \rho_{j,k}\right]  ,$ with $\rho
_{j,j}\allowbreak=\allowbreak1,$ $j,k\allowbreak=\allowbreak1,\ldots,n$. Then,
the conditional distribution of the random vector $\mathbf{X}_{1}%
\allowbreak=\allowbreak(X_{1},\ldots,X_{n-1})^{T}$ given $X_{n}=z$ is normal
$N\left(  \mathbf{\rho}z,\mathbf{\Sigma}_{1}\right)  ,$ where the
$n-1-$dimensional vector $\mathbf{\rho}$ has entries $\rho_{k,n},$
$k\allowbreak=\allowbreak1,\ldots,n-1$ .and matrix $\mathbf{\Sigma}%
_{1}\allowbreak=\allowbreak\left[  \rho_{k,j}-\rho_{k,n}\rho_{j,n}\right]  $,
$k,j\allowbreak=\allowbreak1,\ldots,n-1$.
\end{proposition}

\begin{proof}
Follows immediately (\ref{margin}) applied with $\mathbf{\Sigma}%
_{11}\allowbreak=\allowbreak\lbrack\rho_{k,j}]_{k,j=1,\ldots,n-1}%
,\mathbf{\Sigma}_{21}\allowbreak=\allowbreak\mathbf{\rho.}$
\end{proof}

Recall also that the family of monic polynomials $\left\{  H_{n}(x)\right\}
_{n\geq-1}$ (i.e., having $1$ as the coefficient by the monomial of the
highest order) that are orthogonal with respect to the density $f_{N}$ are
defined by the following three-term recurrence
\[
H_{n+1}(x)=xH_{n}(x)-nH_{n-1}(x),
\]
with $H_{-1}(x)\allowbreak=\allowbreak0$ and $H_{0}(x)\allowbreak
=\allowbreak1.$ They are called probabilistic Hermite polynomials. Let us
collect several well-known facts concerning these polynomials. We have

\begin{lemma}
\label{Her}a)
\begin{equation}
\frac{1}{\sqrt{2\pi}}\int_{-\infty}^{\infty}H_{n}(x)H_{m}(x)e^{-x^{2}/2}dx=%
\begin{cases}
0\text{,} & \text{if }n\neq m\text{;}\\
n!\text{,} & \text{if .}n=m
\end{cases}
. \label{ortH}%
\end{equation}

b) The density of the bivariate normal distribution $N\left(  \left[
\begin{array}
[c]{c}%
0\\
0
\end{array}
\right]  ,\left[
\begin{array}
[c]{cc}%
1 & \rho\\
\rho & 1
\end{array}
\right]  \right)  $, i.e., function
\[
f_{2N}(x,y|\rho)=\frac{1}{2\pi\sqrt{(1-\rho^{2})}}\exp\left(  -\frac
{x^{2}-2\rho xy+y^{2}}{2(1-\rho^{2})}\right)  ,
\]
can be expanded in the following series (so-called Mehler expansion)%
\begin{equation}
f_{2N}(x,y|\rho)=f_{N}(x)f_{N}(y)\sum_{n\geq0}\frac{\rho^{n}}{n!}H_{n}%
(x)H_{n}(y), \label{Meh}%
\end{equation}
convergent for all $x,y\in\mathbb{R}$ and $\left\vert \rho\right\vert <1.$

c) For all non-negative integers $m,k$ and $\left\vert \rho\right\vert <1$ ,
$x,y\in\mathbb{R}$
\begin{equation}
\sum_{j\geq0}\frac{\rho^{j}}{j!}H_{j+m}(x)H_{j+k}(y)=Q_{m,k}(x,y|\rho
)\sum_{j\geq0}\frac{\rho^{j}}{j!}H_{j}(x)H_{j}(y), \label{Carlitz}%
\end{equation}
where
\begin{align}
Q_{m,k}(x,y|\rho)  &  =\sum_{s=0}^{k}\binom{k}{s}(-\rho)^{s}H_{k-s}%
(y)\frac{H_{m+s}\left(  \frac{x-\rho y}{\sqrt{1-\rho^{2}}}\right)  }{\left(
1-\rho^{2}\right)  ^{\frac{m+s}{2}}}.\label{Car2}\\
& \nonumber
\end{align}

d) We have also
\begin{align*}
Q_{m,k}(x,y|\rho)  &  =Q_{k,m}(y,x|\rho),\\
Q_{m,0}(x,y,\rho)  &  =H_{m}\left(  \frac{x-\rho y}{\sqrt{1-\rho^{2}}}\right)
/\left(  1-\rho^{2}\right)  ^{m/2},
\end{align*}

e)
\begin{align}
x^{n}  &  =\sum_{k=0}^{\left\lfloor n/2\right\rfloor }\frac{n!}{2^{k}%
k!(n-2k)!}H_{n-2k}(x),\label{xnaH}\\
H_{n}(x)  &  =\sum_{j=0}^{\left\lfloor n/2\right\rfloor }\frac{(-1)^{j}%
n!}{2^{j}j!(n-2j)!}x^{n-2j}. \label{Hnax}%
\end{align}
We will use the following modification of this formula
\[
x^{n}=a^{n}\left(  \frac{x}{a}\right)  ^{n}=a^{n}\sum_{k=0}^{\left\lfloor
n/2\right\rfloor }\frac{n!(a^{2})^{k}}{2^{k}k!(n-2k)!}a^{n-2k}H_{n-2k}(x/a).
\]
f) We will also use the following formula, known as the umbral identity%
\begin{align}
&  (a^{2}+b^{2})^{n/2}H_{n}\left(  \frac{ay+bz}{\sqrt{a^{2}+b^{2}}}\right)
\label{H(sum)}\\
&  =\sum_{k=0}^{n}\binom{n}{k}a^{k}H_{k}(y)b^{n-k}H_{n-k}(z).\nonumber
\end{align}

g) Now, to proceed further we will use the following linearization formula%
\begin{equation}
H_{n}(x)H_{m}(x)=\sum_{j=0}^{\min(n,m)}\binom{n}{j}\binom{m}{j}j!H_{n+m-2j}%
\left(  x\right)  . \label{Lin}%
\end{equation}

h) We also have the following Taylor expansion formula
\[
H_{n}(x+y)=\sum_{j=0}^{n}\binom{n}{j}x^{j}H_{n-j}(y).
\]

i) We also have the multiplication theorem%
\begin{equation}
H_{n}(\sigma z)=\sum_{j=0}^{\left\lfloor n/2\right\rfloor }\binom{n}{2j}%
\frac{(2j)!}{2^{j}j!}\sigma^{n-2j}(\sigma^{2}-1)^{j}H_{n-2j}(x). \label{MT}%
\end{equation}

The following result follows immediately from this formula.%
\begin{align*}
&  \frac{1}{\sqrt{2\pi}}\int_{-\infty}^{\infty}H_{n}(\sigma x)\exp\left(
-\frac{x^{2}}{2}\right)  dx\\
&  =%
\begin{cases}
0\text{,} & \text{if }n\text{ is odd;}\\
\frac{(2m)!(\sigma^{2}-1)^{m}}{2^{m}m!}\text{,} & \text{if }n=2m\text{.}%
\end{cases}
.
\end{align*}

j)
\begin{equation}
\sum_{s=0}^{n}\binom{n}{s}H_{s}(z)(i)^{n-s}H_{n-s}(iz)=\delta_{0,n}
\label{splotH}%
\end{equation}

\end{lemma}

\begin{proof}
Proofs of formulas b), e), f), g), h), and i) can be found in standard books on
classical orthogonal polynomials such as, e.g., \cite{Andrews1999} or
\cite{IA} or simply Wikipedia. For the less obvious formula (\ref{splotH})
recall that it follows almost directly from the fact that $\sum_{n\geq0}%
H_{n}(x)\frac{t^{n}}{n!}\allowbreak=\allowbreak\exp(xt-t^{2}/2)$ and
$\sum_{n\geq0}H_{n}(ix)\frac{(it)^{n}}{n!}\allowbreak=\allowbreak
\exp(-xt+t^{2}/2),$ hence
\[
\sum_{n\geq0}\frac{t^{n}}{n!}\left(  \sum_{s=0}^{n}\binom{n}{s}H_{s}%
(z)(i)^{n-s}H_{n-s}(iz)\right)  =\exp(zt-t^{2}/2-zt+t^{2}/2)=1.
\]

Statements c) and d) follow first from the unnumbered formula at the top of page 639 of
\cite{SzablAW} and further Lemma 3 i) and Corollary 3 of \cite{SzablAW} with
$q\allowbreak=\allowbreak1.$
\end{proof}

As an immediate corollary we have the following two simple observations:

\begin{corollary}
1.%
\[
\frac{x^{2}-2\rho xy+y^{2}}{2(1-\rho^{2})}-\frac{y^{2}}{2}=\frac{\left(
x-\rho y\right)  ^{2}}{2(1-\rho^{2})},
\]
consequently, we see that
\[
\int f_{CN}(x|y,\rho)f_{N}(y)dy=f_{N}(x).
\]
We used here denotation (\ref{cn}).

2. Using umbral identity we get
\begin{align*}
H_{n}\left(  \frac{x-\rho y}{\sqrt{1-\rho^{2}}}\right)  /\left(  1-\rho
^{2}\right)  ^{n/2}  &  =\sum_{s=0}^{n}\binom{n}{s}(-\rho)^{s}H_{n-s}(x)\\
&  \times H_{s}\left(  \frac{y-\rho x}{\sqrt{1-\rho^{2}}}\right)  /\left(
1-\rho^{2}\right)  ^{s/2}.
\end{align*}

\end{corollary}

Using this lemma, we will derive auxiliary formulas, which we will compile
into the Proposition below.

\begin{proposition}
\label{AHer}For all $m,a\in\mathbb{R}$ and positive $\sigma$ we have

1.%
\begin{gather*}
d_{n,k}(m,\sigma)=\frac{1}{\sigma\sqrt{2\pi}}\int_{-\infty}^{\infty}%
H_{n}(x)H_{k}(x)\exp\left(  -\frac{(x-m)^{2}}{2\sigma^{2}}\right)  dx\\
=\sum_{j=0}^{\min(n,k)}\binom{n}{j}\binom{k}{j}j!\sum_{s=0}^{\left\lfloor
(n+k)/2\right\rfloor -j}\frac{(n+k-2j)!a^{n+k-2j-2k}}{2^{s}(n+k-2j-2s)}\\
\times(\sigma^{2}-1+a^{2})^{s}H_{n+k-2j-2s}(m/a).
\end{gather*}

2. In particular for $k\allowbreak=\allowbreak0$ we get%
\begin{align*}
&  \frac{1}{\sigma\sqrt{2\pi}}\int_{-\infty}^{\infty}H_{n}(x)\exp\left(
-\frac{(x-m)^{2}}{2\sigma^{2}}\right)  dx\\
&  =\sum_{s=0}^{\left\lfloor n/2\right\rfloor }\frac{n!a^{n-2s}}%
{2^{s}(n-2s)!s!}H_{n-2s}(m/a)(\sigma^{2}+a^{2}-1)^{s}.
\end{align*}

If we take $\sigma\allowbreak=\allowbreak\sqrt{1-\rho^{2}},$ $m\allowbreak
=\allowbreak\rho y$ and $a\allowbreak=\allowbreak\rho$ we get $\sigma^{2}%
+\rho^{2}-1\allowbreak=\allowbreak0$ and consequently%
\begin{equation}
\frac{1}{\sqrt{2\pi(1-\rho^{2})}}\int_{-\infty}^{\infty}H_{n}(x)\exp\left(
-\frac{(x-\rho y)^{2}}{2(1-\rho^{2})}\right)  dx=\rho^{n}H_{n}(y).
\label{Cond}%
\end{equation}

3.
\[
\frac{1}{\sqrt{2\pi}}\int_{-\infty}^{\infty}H_{n}(\sigma x)\exp\left(
-\frac{x^{2}}{2}\right)  dx=%
\begin{cases}
0\text{,} & \text{if }n\text{ is odd;}\\
\frac{(2m)!}{2^{m}m!}\left(  \sigma^{2}-1\right)  ^{m}\text{,} & \text{if
}n\allowbreak=\allowbreak2m\text{.}%
\end{cases}
\]

4. If $\left[
\begin{array}
[c]{c}%
X\\
Y
\end{array}
\right]  \sim N\left(  \left[
\begin{array}
[c]{c}%
0\\
0
\end{array}
\right]  ,\left[
\begin{array}
[c]{cc}%
1 & \rho\\
\rho & 1
\end{array}
\right]  \right)  ,$ then%
\[
E\left(  H_{n}(X)H_{m}(Y)\right)  =%
\begin{cases}
0\text{,} & \text{if }n\neq m\text{;}\\
n!\rho^{n}\text{,} & \text{if }n\allowbreak=\allowbreak m\text{.}%
\end{cases}
\]

5.
\begin{equation}
H_{n}(\left(  X-m\right)  /\sigma)=\frac{1}{\sigma^{n}}\sum_{s=0}^{n}%
(-1)^{s}\binom{n}{s}H_{n-s}(x)\sum_{j=0}^{\left\lfloor s/2\right\rfloor }%
\frac{m^{s-2j}(1-\sigma^{2})^{j}}{2^{j}j!(s-2j)!}, \label{Hn1}%
\end{equation}
and in particular%
\begin{equation}
H_{n}\left(  \frac{x-\rho z}{\sqrt{1-\rho^{2}}}\right)  =\frac{1}{\left(
1-\rho^{2}\right)  ^{n/2}}\sum_{s=0}^{n}(i\rho)^{s}\binom{n}{s}H_{n-s}%
(x)H_{s}(iz). \label{Hn2}%
\end{equation}

6) Let us define the integer functions defined by
\begin{align*}
r_{k_{1},\ldots,k_{n}}  &  =E\prod_{j=1}^{n}H_{k_{j}}(X)\\
&  =\int_{\mathbb{R}}\left(  \prod_{j=1}^{n}H_{k_{j}}(z)\right)  f_{N}(z)dz,
\end{align*}
where $X$ is normal random variable $N(0,1).$ The generating function if this
function is equal to
\begin{equation}
\sum_{k_{1}}^{\infty}\ldots\sum_{k_{n}}^{\infty}r_{k_{1},\ldots,k_{n}}%
\prod_{j=1}^{n}\frac{t_{j}^{k_{j}}}{k_{j}!}=\exp\left(  \frac{1}{2}\left(
\sum_{j=1}^{n}t_{j}\right)  ^{2}-\frac{1}{2}\sum_{j=1}^{n}t_{j}^{2}\right)  .
\label{genr}%
\end{equation}
In particular we have
\[
r_{n}\allowbreak=\allowbreak\delta_{0,n},~~r_{n,m}\allowbreak=\allowbreak%
\begin{cases}
0\text{,} & \text{if }n\neq m\text{;}\\
n!\text{,} & \text{if }n=m\text{. }%
\end{cases}
,
\]

\begin{equation}
r_{n,m,k}=\left\{
\begin{tabular}
[c]{ll}%
$0,$ &
\begin{tabular}
[c]{l}%
If $n+k+m$ is odd or\\
$m>n+k$ or $n>k+m$ or $k>m+n$%
\end{tabular}
\\
$\frac{m!n!k!}{(\frac{n+k-m}{2})!\left(  \frac{n+m-k}{2}\right)  !\left(
\frac{m+k-n}{2}\right)  !},$ & $\text{otherwise.}$%
\end{tabular}
\right.  , \label{r3}%
\end{equation}
\newline\newline

Assuming $n\geq m\geq k\geq j$ we have $r_{n,m,k,j}\allowbreak=\allowbreak0 $
if $n+m+k+j$ is odd and%

\begin{align}
r_{n,m,k,j}  &  =\sum_{s=0}^{j}\binom{n}{s+\frac{n+m-j-k}{2}}\binom{m}%
{s+\frac{n+m-j-k}{2}}\binom{k}{s}\times,\label{r41}\\
&  \binom{j}{s}s!(s+\frac{n+m-j-k}{2})!(k+j-2s)!\nonumber
\end{align}

Notice also that the sequence $r_{n,n,n,n}$ divided by $n!^{2}$ i.e., the
sequence $\left\{  \sum_{j=0}^{n}\binom{n}{j}^{2}\binom{2j}{j}\right\}
_{n\geq0}$ is the A002893 sequence of OESIS, while $r_{n,n,n-1,n-1}$ divided
by $n!(n-1)!$ for $k\allowbreak=\allowbreak j\allowbreak=\allowbreak n-1$,
i.e., the sequence $\left\{  \sum_{j=0}^{n-1}\binom{n}{j}\binom{n-1}{j}%
\binom{2j}{j}\right\}  _{n\geq1}$ is the A087457 sequence in OESIS.
\end{proposition}

\begin{proof}
From the formula (\ref{MT}) immediately follow the following result 3.%
\begin{align*}
&  \frac{1}{\sqrt{2\pi}}\int_{-\infty}^{\infty}H_{n}(\sigma x)\exp\left(
-\frac{x^{2}}{2}\right)  dx\\
&  =%
\begin{cases}
0\text{,} & \text{if }n\text{ is odd;}\\
\frac{(2m)!(\sigma^{2}-1)^{m}}{2^{m}m!}\text{,} & \text{if }n=2m\text{.}%
\end{cases}
\end{align*}

Now to get 2. we proceed as follows
\begin{gather*}
\frac{1}{\sigma\sqrt{2\pi}}\int_{-\infty}^{\infty}H_{n}(x)\exp\left(
-\frac{(x-m)^{2}}{2\sigma^{2}}\right)  dx\\
=\frac{1}{\sqrt{2\pi}}\int_{-\infty}^{\infty}H_{n}(\sigma x+m)\exp\left(
-\frac{x^{2}}{2}\right)  dx\\
=\frac{1}{\sqrt{2\pi}}\int_{-\infty}^{\infty}\sum_{j=0}^{n}\binom{n}{j}%
m^{n-j}H_{j}(\sigma y)\exp\left(  -\frac{x^{2}}{2}\right)  dx\\
=\sum_{j=0}^{\left\lfloor n/2\right\rfloor }\frac{n!}{2^{j}j!(n-2j)!}%
m^{n-2j}(\sigma^{2}-1)^{j}.
\end{gather*}
Further we have
\begin{align*}
&  \frac{1}{\sigma\sqrt{2\pi}}\int_{-\infty}^{\infty}H_{n}(x)\exp\left(
-\frac{(x-m)^{2}}{2\sigma^{2}}\right)  dx\\
&  =\sum_{j=0}^{\left\lfloor n/2\right\rfloor }\frac{n!a^{n-2j}}%
{2^{j}j!(n-2j)!}(m/a)^{n-2j}(\sigma^{2}-1)^{j}\\
&  =\sum_{j=0}^{\left\lfloor n/2\right\rfloor }\frac{n!a^{n-2j}}%
{2^{j}j!(n-2j)!}(\sigma^{2}-1)^{j}\sum_{k=0}^{\left\lfloor n/2\right\rfloor
-j}\frac{(n-2j)!}{2^{k}k!(n-2j-2k)!}H_{n-2j-2k}(m/a)\\
&  =\sum_{s=0}^{\left\lfloor n/2\right\rfloor }\frac{n!a^{n-2s}}%
{2^{s}(n-2s)!s!}H_{n-2s}(m/a)\sum_{j=0}^{s}\binom{s}{j}(\sigma^{2}%
-1)^{j}a^{2(s-j)}\\
&  =\sum_{s=0}^{\left\lfloor n/2\right\rfloor }\frac{n!a^{n-2s}}%
{2^{s}(n-2s)!s!}H_{n-2s}(m/a)(\sigma^{2}+a^{2}-1)^{s}.
\end{align*}

To get 1. we apply (\ref{Lin}) and the proceed as in 2.

4. We proceed as follows. From (\ref{Meh}) we deduce that the conditional
density of $X|Y=y$ can we expanded in the series
\[
\frac{1}{\sqrt{2\pi}}\exp(-x^{2}/2)\sum_{j\geq0}\frac{\rho^{j}}{j!}%
H_{j}(x)H_{j}(y).
\]
Hence,
\[
E(H_{n}(X)|Y=y)=\rho^{n}H_{n}(y).
\]

5. Applying assertion f of Lemma \ref{Her} and then assertion g we get%
\begin{gather*}
H_{n}(\left(  X-m\right)  /\sigma)=\sum_{k=0}^{n}\binom{n}{k}(m/\sigma
)^{n-k}H_{k}(x/\sigma)\\
\frac{1}{_{\sigma^{n}}}\sum_{k=0}^{n}\frac{(-1)^{n-k}m^{n-k}n!}{(n-k)!}%
\sum_{j=0}^{\left\lfloor k/2\right\rfloor }\frac{(1-\sigma^{2})^{j}}%
{2^{j}j!(k-2j)!}H_{k-2j}(x).
\end{gather*}
Now it remains to change the order of summation. To get (\ref{Hn2}) we insert
$\sigma\allowbreak=\allowbreak\sqrt{1-\rho^{2}}$ and $m\allowbreak=\allowbreak
z\rho.$ Then we have to notice that n%
\[
\sum_{j=0}^{\left\lfloor s/2\right\rfloor }\frac{m^{s-2j}(1-\sigma^{2})^{j}%
}{2^{j}j!(s-2j)!}=\frac{\rho^{s}}{s!}\sum_{j=0}^{\left\lfloor s/2\right\rfloor
}\frac{s!z^{s-2j}}{2^{j}j!(s-2j)!}=\frac{i^{s}\rho^{s}}{s!}H_{s}(iz),
\]
by (\ref{Hnax})

6. Using formula (18.18.11) of \cite{Nist} we easily deduce that we have
\begin{align*}
\sum_{m_{1}\geq0}\ldots\sum_{m_{n}\geq0}\prod_{j=1}^{n}\frac{a_{j}^{m_{j}}%
}{m_{j}!}H_{m_{j}}(x)  &  =\sum_{n\geq0}\frac{\left(  \sum_{j=1}^{n}a_{j}%
^{2}\right)  ^{n/2}}{n!}H_{n}\left(  \frac{x\sum_{j=1}^{n}a_{j}}{\sqrt
{\sum_{j=1}^{n}a_{j}^{2}}}\right) \\
&  =\exp\left(  x\frac{\sum_{j=1}^{n}a_{j}}{\sqrt{\sum_{j=1}^{n}a_{j}^{2}}%
}\sqrt{\sum_{j=1}^{n}a_{j}^{2}}-\frac{1}{2}\sum_{j=1}^{n}a_{j}^{2}\right)  .
\end{align*}
Now we have to integrate this function times $f_{N}(x)$ with respect to $x$.
We get then%
\begin{align*}
&  \int\exp\left(  x\frac{\sum_{j=1}^{n}a_{j}}{\sqrt{\sum_{j=1}^{n}a_{j}^{2}}%
}\sqrt{\sum_{j=1}^{n}a_{j}^{2}}-\frac{1}{2}\sum_{j=1}^{n}a_{j}^{2}\right)
f_{N}(x)dx\\
&  =\exp\left(  -\frac{1}{2}\sum_{j=1}^{n}a_{j}^{2}\right)  \frac{1}%
{\sqrt{2\pi}}\int\exp\left(  x\sum_{j=1}^{n}a_{j}-x^{2}/2\right)  dx\\
&  =\exp\left(  -\frac{1}{2}\sum_{j=1}^{n}a_{j}^{2}\right)  \exp\left(
\left(  \sum_{j=1}^{n}a_{j}\right)  ^{2}/2\right)  .
\end{align*}
That is the formula (\ref{genr}). Using (\ref{Lin}) we have%
\[
r_{n,m,k}=\sum_{j=0}^{\min(n,m)}\binom{n}{j}\binom{m}{j}j!\int H_{n+m-2k}%
(x)H_{k}(x)f_{N}(x)dx.
\]
Now notice that
\[
E\left(  H_{n+m-2j}(Z)H_{k}(Z)\right)  =%
\begin{cases}
0\text{,} & \text{if }k\neq n+m-2j\text{;}\\
k!\text{,} & \text{if }k=n+m-2j\text{.}%
\end{cases}
\]
that is this quantity is non-zero if $n+m+k$ is even and $j\allowbreak
=\allowbreak(n+m-k)/2.$ Notice also that then $m-j\allowbreak=\allowbreak
(m+k-n)/2$ and $n-j\allowbreak=\allowbreak(n+k-m)/2.$

Hence, we have (\ref{r3}). Using (\ref{Lin}) twice we easily show (\ref{r41}).
\end{proof}

\section{Main results \label{GmN}}

We start with the following general simple observation.

\begin{lemma}
Let the $n-$dimensional random vector $\mathbf{X\allowbreak=\allowbreak(}%
X_{1},\ldots,X_{n})^{T}$ be normaly distributed with $N(\mathbf{0,\Sigma)}$,
with positive definite matrix $\mathbf{\Sigma\allowbreak=\allowbreak}\left[
\rho_{j,k}\right]  $ with $\rho_{j,j}\allowbreak=\allowbreak1,$
$j,k\allowbreak=\allowbreak1\ldots,n$. This means that all marginal
$1-$dimensional distribution are $N(0,1)$.

1. The density of this distribution can be expanded in the following infinite
series
\begin{equation}
f_{mN}(x_{1}\ldots,x_{n}|\mathbf{\Sigma)=}\prod_{j=1}^{n}f_{N}(x_{j}%
)\sum_{k_{1},\ldots,k_{n}}c_{k_{1},\ldots,k_{n}}\prod_{j=1}^{n}H_{k_{j}}%
(x_{j}). \label{rozkl}%
\end{equation}

2. The coefficients of this expansion are given by equation%
\begin{equation}
c_{k_{1},\ldots,k_{n}}\prod_{j=1}^{n}k_{j}!=E\left(  \prod_{j=1}^{n}H_{k_{j}%
}(X_{j})\right)  \label{coef}%
\end{equation}

3. Moreover, the coefficients $c_{k_{1},\ldots,k_{n}}$ take on non-zero values
if $\sum_{j=1}^{n}k_{j}$ is even and for all $j\allowbreak=\allowbreak
1,\ldots,n$%
\begin{equation}
k_{j}\leq\sum_{s\neq j}k_{s}. \label{PP}%
\end{equation}

\end{lemma}

\begin{proof}
1. The fact that expansion \ref{rozkl} is possible follows the fact that the
following integral
\[
\underset{\mathbb{R}^{n}}{\int\ldots\int}f_{mN}^{2}(x_{1}\ldots,x_{n}%
|\mathbf{\Sigma)}dx_{1}\ldots dx_{n}%
\]
is finite, This means that the density $f_{mN}$ belongs to the space of square
integrable functions $L_{2}(\mathbb{R}^{n},\mathbf{\mu}_{n}),$ where,
$\mathbf{\mu}_{n}$ denotes $n-$dimensional product measure of normal $N\left(
0,1\right)  $ measures. The polynomials $\prod_{j=1}^{n}H_{k_{j}}(x_{j}),$
$k_{1},\ldots,k_{n}\geq0$ constitute a base in this space. Hence, the
(\ref{rozkl}) exists.

2. Let us multiply both sides of (\ref{rozkl}) by $\prod_{j=1}^{n}H_{k_{j}%
}(x_{j})$ for some fixed $k_{1},\ldots,k_{n}.$ and the let us integrate the
product over $\mathbb{R}^{n}$. Since Hermite polynomials are orthogonal with
respect to $f_{N}$, we get the desired identity.

3 Let us recall two specific features of the multivariate normal distribution
. Namely, that all its marginal distribution are also normal and that all its
conditional moments of order say $n$ are almost surely polynomial of order not
exceeding $n.$ Hence, we deduce that say
\[
E\left(  \prod_{j=1}^{n-1}H_{k_{j}}(X_{j})|X_{n}\right)  ,
\]
is a polynomial $P_{s}(X_{n})$ in $X_{n}$ of order not exceeding
$s\allowbreak=\allowbreak\sum_{j=1}^{n-1}k_{j}.$ Consequently, we must have
$EP_{s}(X_{n})H_{l}(X_{n})\allowbreak=\allowbreak0,$ whenever $l>s$. Let us
recall the Isserlis Theorem which states that the moments of multivariate
normal distribution of odd order are equal to zero. \ Also observe that
Hermite polynomials of odd order consist solely of odd powers of the variable,
whereas those of even order are made up exclusively of even powers. Hence, the
product $\prod_{j=1}^{n}H_{k_{j}}(j_{j})$ with $\sum_{j=1}^{n}k_{j}$ being an
odd integer contains only product of monomials $\prod_{j=1}^{n}x_{j}^{s_{j}}$
of odd order only. But by the Isserlis Theorem we have $E\prod_{j=1}^{n}%
X_{j}^{s_{j}}\allowbreak=\allowbreak0$ and consequently $E\prod_{j=1}%
^{n}H_{s_{j}}\left(  X_{j}\right)  \allowbreak=\allowbreak0$ if $\sum
_{j=1}^{n}s_{j}\allowbreak$ is odd.
\end{proof}

Let us introduce some auxiliary sequences of numbers and functions that will
be of use in the squeal.

1. The sequence of polynomials in $z$ and entries of the matrix
$\mathbf{\Sigma,}$ defined by the relationship
\[
h_{k_{1}\ldots,k_{n-1}}(\mathbf{\Sigma,}z\mathbf{)\allowbreak=\allowbreak
}E\left(  \prod_{j=1}^{n-1}H_{k_{j}}(X_{j})|X_{n}=z\right)  .
\]

2. The sequence of polynomials of the entries of the matrix $\mathbf{\Sigma
,}$
\[
C_{j_{1},\ldots,j_{k}}^{(k)}(\mathbf{\Sigma)\allowbreak=\allowbreak}%
E\prod_{s=1}^{k}H_{j_{s}}(X_{s}).
\]
Note that if $\mathbf{\Sigma}$ has $1^{\prime}s$ on its diagonal
then
\begin{align*}
C_{j_{1},\ldots,j_{k}}^{(k)}(\mathbf{\Sigma)}  &  \mathbf{=}E\left(
E(\prod_{j=1}^{k-1}H_{j_{s}}(X_{s})|X_{k})H_{j_{k}}(X_{k})\right) \\
&  =\int h_{j_{1}\ldots j_{k-1}}(\mathbf{\Sigma,}z)H_{j_{k}}(z)f_{N}(x)dx
\end{align*}

We have the following general result.

\begin{theorem}
\label{COND}Let $n-$dimensional random vector be normal with $N\left(
\mathbf{0},\mathbf{\Sigma}\right)  ,$ where $\mathbf{\Sigma\allowbreak
=\allowbreak\lbrack}\rho_{j.k}]$ with $\rho_{j,j}\allowbreak=\allowbreak1.$

By $\mathbf{\rho}$ let us denote the $n-1$- dimensional vector with entries
$\left(  \rho_{1,n},\ldots,\rho_{n-1,n}\right)  $ and by $\mathbf{\Sigma}_{c}$
the $(n-1)\times(n-1)$ matrix defined by $[\rho_{k,j}-\rho_{k,n}\rho
_{j,n}]_{j,k=1,\ldots,n-1}$. Notice that the diagonal of this matrix takes on
values $1-\rho_{i,n}^{2}$ for $i\allowbreak=\allowbreak1,\ldots,n-1.$ Let us
denote for simplicity
\begin{equation}
\hat{C}_{k_{1},\ldots k_{n-1}}\allowbreak=\allowbreak C_{k_{1},\ldots,k_{n-1}%
}^{(n-1)}(\mathbf{D\Sigma}_{c}\mathbf{D})\prod_{j=1}^{n-1}\left(  \sqrt
{1-\rho_{j,n}^{2}}\right)  ^{k_{j}}, \label{CC}%
\end{equation}
for $k_{j}\geq0,$ $j=1,\ldots,n-1$, where $\mathbf{D}$ is an $(n-1)\times
(n-1)$ diagonal matrix with $1/\sqrt{1-\rho_{j,n}^{2}}$ as an $j-th$ element
of the diagonal.

Then the conditional expectations%
\[
h_{k_{1}\ldots,k_{n-1}}(\mathbf{\Sigma,}z\mathbf{)\allowbreak=\allowbreak
}E\left(  \prod_{j=1}^{n-1}H_{k_{j}}(X_{j})|X_{n}=z\right)
\]
are equal to
\begin{equation}
h_{k_{1}\ldots,k_{n-1}}(\mathbf{\Sigma,}z\mathbf{)\allowbreak}=\sum_{s_{1}%
=0}^{k_{1}}\ldots\sum_{s_{n-1}=0}^{k_{n-1}}\hat{C}_{s_{1},\ldots s_{n-1}}%
\prod_{j=1}^{n-1}\binom{k_{j}}{s_{j}}\rho_{j,n}^{k_{j}-s_{j}}H_{k_{j}-s_{j}%
}(z), \label{E|z}%
\end{equation}
Moreover, polynomials $h_{k_{1}\ldots,k_{n-1}}$ satisfy the following system
of equations:%
\begin{equation}
\sum_{s_{1}=0}^{k_{1}}\ldots\sum_{s_{n-1}=0}^{k_{n-1}}h_{s_{1}\ldots,s_{n-1}%
}(\mathbf{\Sigma,}z\mathbf{)}\prod_{j=1}^{n-1}\binom{k_{j}}{s_{j}}\left(
i\rho_{j,n}\right)  ^{k_{j}-s_{j}}H_{k_{j}-s_{j}}(iz)=\hat{C}_{k_{1}%
,\ldots,k_{n-1}}, \label{CMom}%
\end{equation}

\end{theorem}

\begin{proof}
First of all, recall that the conditional distribution of $\left(
X_{1},\ldots,X_{n-1}\right)  $ given $X_{n}\allowbreak=\allowbreak z$ has
normal distribution $N\left(  \mathbf{\rho}z,\mathbf{\Sigma}_{c}\right)  .$
Secondly, recall that following (\ref{afine}) is a random vector with the
following
\[
\left(  \frac{x_{1}-z\rho_{1,n}}{\sqrt{1-\rho_{1n}^{2}}},\ldots,\frac
{x_{n-1}-z\rho_{n-1,n}}{\sqrt{1-\rho_{n-1,n}^{2}}}\right)  ,
\]
coordinates, has $N\left(  \mathbf{0,D\Sigma}_{c}\mathbf{D}\right)  $
distribution. Consequently, we deduce that
\[
E\prod_{j=1}^{n-1}H_{k_{j}}\left(  \frac{X_{j}-\rho_{j,n}z}{\sqrt{1-\rho
_{j,n}^{2}}}\right)  =C_{k_{1},\ldots,k_{n-1}}^{(n-1)}(\mathbf{D\Sigma}%
_{c}\mathbf{D}).
\]
Thirdly, recall that for each $j\allowbreak=\allowbreak1,\ldots,n-1$ we have%
\[
H_{k_{j}}\left(  \frac{X_{j}-\rho_{j,n}z}{\sqrt{1-\rho_{j,n}^{2}}}\right)
=\frac{1}{\left(  \sqrt{1-\rho_{j,n}^{2}}\right)  ^{k_{j}}}\sum_{s=1}^{k_{j}%
}\binom{k_{j}}{s}H_{s}(X_{j})\left(  i\rho_{j,n}\right)  ^{k_{j}-s}H_{k_{j}%
-s}\left(  iz\right)  .
\]
Thus, we have%
\begin{align*}
&  \sum_{s_{1}=0}^{k_{1}}\ldots\sum_{s_{n-1}=0}^{k_{n-1}}h_{k_{1}%
\ldots,k_{n-1}}(\mathbf{\Sigma,}z\mathbf{)}\prod_{j=1}^{n-1}\binom{k_{j}%
}{s_{j}}\left(  i\rho_{j,n}\right)  ^{k_{j}-s_{j}}H_{k_{j}-s_{j}}(iz)\\
&  =C_{k_{1},\ldots,k_{n-1}}(\mathbf{D\Sigma}_{c}\mathbf{D})\prod_{j=1}%
^{n-1}\left(  \sqrt{1-\rho_{j,n}^{2}}\right)  ^{k_{j}}=\hat{C}_{k_{1}%
,\ldots,k_{n-1}}.
\end{align*}
Thus, we got the system of equations (\ref{CMom}). We will show that functions
(\ref{E|z}) satisfy this system of equations. Namely, we get after inserting
(\ref{E|z}) into (\ref{CMom})
\begin{align*}
&  \sum_{s_{1}=0}^{k_{1}}\ldots\sum_{s_{n-1}=0}^{k_{n-1}}\sum_{m_{1}}^{s_{1}%
}\ldots\sum_{m_{n-1}}^{s_{n-1}}\hat{C}_{m_{1},\ldots m_{n-1}}\\
&  \times\prod_{j=1}^{n-1}\binom{s_{j}}{m_{j}}\left(  \rho_{jn}\right)
^{s_{j}-m_{j}}H_{s_{j}-m_{j}}\prod_{j=1}^{n-1}\binom{k_{j}}{s_{j}}\left(
i\rho_{j,n}\right)  ^{k_{j}-s_{j}}H_{k_{j}-s_{j}}(iz)\\
&  =\sum_{m_{1}=0}^{k_{1}}\ldots\sum_{m_{n-1}=0}^{k_{n-1}}\hat{C}%
_{m_{1},\ldots m_{n-1}}\prod_{j=1}^{n-1}\binom{k_{j}}{m_{j}}\left(  \rho
_{jn}\right)  ^{k_{j}-m_{j}}\\
&  \times\sum_{s_{1}=m_{1}}^{k_{1}}\ldots\sum_{s_{n-1}=m_{n-1}}^{k_{n-1}}%
\prod_{j=1}^{n-1}\binom{k_{j}-m_{j}}{s_{j}-m_{j}}H_{s_{j}-m_{j}}(z)\left(
i\right)  ^{k_{j}-s_{j}}H_{k_{j}-s_{j}}(iz)\\
&  =\sum_{m_{1}=0}^{k_{1}}\ldots\sum_{m_{n-1}=0}^{k_{n-1}}\hat{C}%
_{m_{1},\ldots m_{n-1}}\prod_{j=1}^{n-1}\binom{k_{j}}{m_{j}}\left(  \rho
_{jn}\right)  ^{k_{j}-m_{j}}\\
&  \times\sum_{t_{1}=0}^{k_{1}-m_{1}}\ldots\sum_{t_{n-1}=0}^{k_{n-1}-m_{n-1}%
}\prod_{j=1}^{n-1}\binom{k_{j}-m_{j}}{t_{j}}H_{t_{j}}(z)\left(  i\right)
^{k_{j}-m_{j}-t_{j}{}_{j}}H_{k_{j}-m_{j}-t_{j}}(iz)\\
&  =\sum_{m_{1}=0}^{k_{1}}\ldots\sum_{m_{n-1}=0}^{k_{n-1}}\hat{C}%
_{m_{1},\ldots m_{n-1}}\prod_{j=1}^{n-1}\binom{k_{j}}{m_{j}}\left(  \rho
_{jn}\right)  ^{k_{j}-m_{j}}\\
&  \times\prod_{j=1}^{n-1}\sum_{t_{j}=0}^{k_{j}-m_{j}}\binom{k_{j}-m_{j}%
}{t_{j}}H_{t_{j}}(z)\left(  i\right)  ^{k_{j}-m_{j}-t_{j}{}_{j}}H_{k_{j}%
-m_{j}-t_{j}}(iz).
\end{align*}
Now, recall that $\sum_{s=0}^{n}\binom{n}{s}H_{s}(z)(i)^{n-s}H_{n-s}%
(iz)=\delta_{0,n},$ hence the whole expression equals to $\hat{C}%
_{k_{1},\ldots k_{n-1}}.$
\end{proof}

\begin{remark}
The quantity $\hat{C}_{s_{1},\ldots s_{n-1}}$ defined by the formula
(\ref{CC}) is a polynomial in coefficients $\rho_{jk}$. It does not contain
$\sqrt{1-\rho_{j,n}^{2}},$ $j\allowbreak=\allowbreak1,\ldots,n$. This is so,
since we have
\begin{equation}
H_{n}(x/\sqrt{1-\rho^{2}})\allowbreak=\allowbreak\frac{1}{(1-\rho^{2})^{n/2}%
}\sum_{j=0}^{\left\lfloor n/2\right\rfloor }\frac{n!\rho^{2j}}{2^{j}%
j!(n-2j)!}H_{n-2j}(x), \label{pom1}%
\end{equation}
by (\ref{MT}). Now, recall formula (\ref{CC}). If we define $\mathbf{Y}$ as
the random vector having normal distribution $N(\mathbf{0,\Sigma}_{c}),$ then
vector $\mathbf{Z\allowbreak=\allowbreak}\left(  Y_{1}/\sqrt{1-\rho_{1,n}^{2}%
},\ldots,Y_{n-1}/\sqrt{1-\rho_{1,n}^{2}}\right)  ^{T},$ has normal $N\left(
\mathbf{0,D\Sigma}_{c}\mathbf{D}\right)  $ distribution . Hence,
\[
C_{k_{1},\ldots,k_{n-1}}^{(n-1)}(\mathbf{D\Sigma}_{c}\mathbf{D})=E\prod
_{j=1}^{n-1}H_{k_{j}}(Y_{j}/\sqrt{1-\rho_{j,n}^{2}}).
\]
But due to formula (\ref{pom1}) we see that
\begin{align*}
&  C_{k_{1},\ldots,k_{n-1}}^{(n-1)}(\mathbf{D\Sigma}_{c}\mathbf{D})\prod
_{j=1}^{n-1}\left(  \sqrt{1-\rho_{j,n}^{2}}\right)  ^{k_{j}}\\
&  =\sum_{s_{1}=0}^{\left\lfloor k_{1}/2\right\rfloor }\ldots\sum_{s_{n-1}%
=0}^{\left\lfloor k_{n}/2\right\rfloor }C_{k_{1}-2s_{1}\ldots,k_{n-1}%
-2s_{n-1}}^{(n-1)}(\mathbf{\Sigma}_{c})\prod_{j=1}^{n-1}\frac{k_{j}!\rho
_{j,n}^{2s_{j}}}{2^{s_{j}}s_{j}!(k_{j}-2s_{j})!},
\end{align*}
Hence, it contains only integer powers of coefficients $\rho_{jk}.$
\end{remark}

\begin{remark}
Notice that the coefficients of the expansion (\ref{rozkl}) are now given by
the iterative relationship. Namely, assuming that the random vector
$\mathbf{X\allowbreak=\allowbreak}\left(  X_{1},\ldots,X_{n}\right)  ^{T}\sim
N(0,\mathbf{\Sigma),}$ where $\mathbf{\Sigma}$ is some positive definite
$n\times n$ matrix with $1^{\prime}s$ on the diagonal, we have%
\begin{align*}
c_{k_{1},\ldots,k_{n}}  &  =\frac{1}{\prod_{j=1}^{n}k_{j}!}E\prod_{j=1}%
^{n}H_{k_{j}}(X_{j})\\
&  =\frac{1}{k_{n}!}\sum_{s_{1}=0}^{k_{1}}\ldots\sum_{s_{n-1}=0}^{k_{n-1}%
}\frac{r_{k_{1}-s_{1},\ldots,k_{n-1}-s_{n-1},k_{n}}}{\prod_{j=1}^{n-1}%
(k_{j}-s_{j})!}\frac{\hat{C}_{s_{1},\ldots s_{n-1}}}{\prod_{j=1}^{n-1}s_{j}%
!}\prod_{j=1}^{n-1}\rho_{j,n}^{k_{j}-s_{j}},
\end{align*}

Consequently, we have the following form of the K-S formula (recall that it
refers to $N(\mathbf{0,\Sigma)}$ distribution with $\mathbf{\Sigma}$ having
$1^{\prime}s$ on its diagonal)
\begin{align*}
&  \frac{1}{\sqrt{(2\pi)^{n}\det\mathbf{\Sigma}}}\exp\left(  -\mathbf{x}%
^{T}\mathbf{\Sigma}^{-1}\mathbf{x}/2+\mathbf{x}^{T}\mathbf{x/2}\right) \\
&  =\sum_{k_{n}\geq0}\frac{H_{k_{n}}(x_{n})}{k_{n}!}\sum_{k_{n-1}\geq0}%
\ldots\sum_{k_{1}\geq0}\prod_{j=1}^{n-1}H_{k_{j}}(x_{j})\\
&  \times\sum_{s_{1}=0}^{k_{1}}\ldots\sum_{s_{n-1}=0}^{k_{n-1}}\frac
{r_{k_{1}-s_{1},\ldots,k_{n-1}-s_{n-1},k_{n}}}{\prod_{j=1}^{n-1}(k_{j}%
-s_{j})!}\frac{\hat{C}_{s_{1},\ldots s_{n-1}}}{\prod_{j=1}^{n-1}s_{j}!}%
\prod_{j=1}^{n-1}\rho_{j,n}^{k_{j}-s_{j}}\\
&  =\sum_{k_{n}\geq0}\frac{H_{k_{n}}(x_{n})}{k_{n}!}\sum_{s_{n-1}\geq0}%
\ldots\sum_{s_{1}\geq0}\frac{\hat{C}_{s_{1},\ldots s_{n-1}}}{\prod_{j=1}%
^{n-1}s_{j}!}\\
&  \times\sum_{k_{n-1}\geq s_{n-1}}\ldots\sum_{k_{1}\geq s_{1}}\frac
{r_{k_{1}-s_{1},\ldots,k_{n-1}-s_{n-1},k_{n}}}{\prod_{j=1}^{n-1}(k_{j}%
-s_{j})!}\prod_{j=1}^{n-1}\rho_{j,n}^{k_{j}-s_{j}}\prod_{j=1}^{n-1}H_{k_{j}%
}(x_{j})\\
&  =\sum_{k_{n}\geq0}\frac{H_{k_{n}}(x_{n})}{k_{n}!}\sum_{s_{n-1}\geq0}%
\ldots\sum_{s_{1}\geq0}\frac{\hat{C}_{s_{1},\ldots s_{n-1}}}{\prod_{j=1}%
^{n-1}s_{j}!}\\
&  \times\sum_{l_{n-1}\geq0}\ldots\sum_{l_{1}\geq0}r_{l_{1},\ldots
,l_{n-1},k_{n}}\frac{\prod_{j=1}^{n-1}\rho_{j,n}^{l_{j}}}{\prod_{j=1}%
^{n-1}l_{j}!}\prod_{j=1}^{n-1}H_{s_{j}+l_{j}}(x_{j}).
\end{align*}
This is so since both expansions (\ref{rozkl}) and K-S are expansions of the
density of $N(\mathbf{0,\Sigma)}$ in terms of products of Hermite polynomials.
Hence, they must be identical.
\end{remark}

\subsection{Markov case\label{Mark}}

Let us recall that by a finite Markov chain it is understood an ordered
sequence of random variables $X_{1},\ldots,X_{n}$ that have the property that
for every integrable function $g$ and $j\allowbreak=\allowbreak1,\ldots,n-1$
we have almost surely%
\[
E\left(  g\left(  X_{j+1}\right)  |X_{j},\ldots,X_{1}\right)  =E\left(
g(X_{j+1}|X_{j}\right)  .
\]
This means that all information about the chain's future behavior is contained
in the last state. For the details see for example. \cite{Dynkin12}.
Traditionally, the indices are interpreted as time and the values of random
variables $X_{j},$ $j\allowbreak=\allowbreak1,\ldots,n$ are called states (of
the chain). A consequence of this assumption is the fact that the joint
distribution of the chain is equal to a product of the marginal distribution
times a product of the conditional ones. In the case of normal Markov chains
with zero expectations, it means that the conditional density of
$X_{j+1}|X_{j}=y$ is given by the function
\begin{equation}
f_{CN}(x|y,\rho_{j+1j})=\frac{1}{\sqrt{2\pi(1-\rho_{j,j+1}^{2})}}\exp\left(
-\frac{(x-\rho_{j,j+1}y)^{2}}{2(1-\rho_{j,j+1}^{2})}\right)  . \label{MCN}%
\end{equation}
If we assume that the marginal density of $X_{1}$ is $f_{N}$, then the joint
density of the random vector $(X_{1},\ldots,X_{n})^{T}$ is equal to
\[
f_{N}(x_{1})\prod_{j=1}^{n-1}f_{CN}(x_{j+1}|x_{j},\rho_{j,,j+1}).
\]
From this assumption follows also the symmetry of the so-described Markov
chain. Namely, we have also
\[
E\left(  g(X_{j})|X_{j+1},\ldots,X_{n}\right)  =E\left(  g(X_{j}%
)|X_{j}\right)  ,
\]
almost surely.

This fact results in the form of the variance--covariance matrix
$\mathbf{\Sigma\allowbreak=\allowbreak\lbrack}\sigma_{j,k}]_{j,k,1,\ldots,n}$
which now has the following entries : $\sigma_{j,j}\allowbreak=\allowbreak1$,
$j\allowbreak=\allowbreak1,\ldots,n$ and $\sigma_{j,k}\allowbreak
=\allowbreak\prod_{m=j}^{k-1}\rho_{m,m+1}$, for $j<k$. It is easy to notice
that then $\det\mathbf{\Sigma\allowbreak=\allowbreak}\prod_{j=1}^{n-1}%
(1-\rho_{j,j+1}^{2})$ and that the inverse matrix , i.e., $\mathbf{\Sigma
\allowbreak}^{-1}$, is a bidiagonal matrix. Moreover, from the Markov
assumption it follows also that
\[
E\left(  H_{k}(X_{j})|X_{n},\ldots,X_{j+1},X_{j-1},\ldots,X_{1}\right)
=E\left(  H_{k}(X_{j})|X_{j+1},X_{j-1}\right)  ,
\]
almost surely. For the proof and details see \cite{Szab10}. Now following
formulas from the assertion of Theorem2 (i.e., (3.1) and (3.2)) the assertion
ii) of Lemma3 and (4.1) of \cite{SzablAW} we can state the following theorem.

\begin{theorem}
Assuming that the sequence of random variables $X_{1},\ldots,X_{n}$ forms a Markov
chain with zero expectations, $f_{N}$ as its marginal distributions and
(\ref{MCN}) as conditional distribution of $X_{j+1}|X_{j}=y$ the conditional
expectation $C_{k}(y,z,\rho_{j-1,j},\rho_{j,j+1})=E\left(  H_{k}%
(X_{j})|X_{n},\ldots,X_{j+1}=z,X_{j-1}=y,\ldots,X_{1}\right)  $ takes one of
the following equivalent forms%
\begin{gather}
C_{k}(y,z,\rho_{j-1,j},\rho_{j,j+1})=\sum_{s=0}^{k}\binom{k}{s}\rho
_{j-1,j}^{k-s}\left(  1-\rho_{j-1,j}^{2}\right)  ^{s}\rho_{j,j+1}^{s}%
H_{k-s}(y)\label{CH1}\\
\times H_{s}\left(  \frac{z-\rho_{j-1,j}\rho_{j,j+1}y}{\sqrt{1-\left(
\rho_{j-1,j}\rho_{j,j+1}\right)  ^{2}}}\right)  /\left(  1-\left(
\rho_{j-1,j}\rho_{j,j+1}\right)  ^{2}\right)  ^{s/2},\nonumber
\end{gather}%
\begin{gather}
C_{k}(y,z,\rho_{j-1,j},\rho_{j,j+1})=\frac{1}{\left(  1-\left(  \rho
_{j-1,j}\rho_{j,j+1}\right)  ^{2}\right)  ^{k}}\sum_{s=0}^{\left\lfloor
k/2\right\rfloor }(-1)^{s}\binom{k}{2s}\binom{2s}{s}s!\nonumber\\
\times\left(  \rho_{j-1,j}\rho_{j,j+1}\right)  ^{2s}\left(  \left(
1-\rho_{j-1,j}^{2}\right)  \left(  1-\rho_{j,j+1}^{2}\right)  \right)  ^{s}\\
\times\sum_{m=0}^{k-2s}\binom{k-2s}{m}\left(  1-\rho_{j-1,j}^{2}\right)
^{m}\left(  1-\rho_{j,j+1}^{2}\right)  ^{k-2s-m}\rho_{j-1,j}^{k-2s-m}%
\rho_{j,j+1}^{m}H_{m}(z)H_{k-2s-m}(y).\nonumber
\end{gather}%
\begin{gather}
C_{k}(y,z,\rho_{j-1,j},\rho_{j,j+1})=\left(  \frac{\rho_{j-1,j}^{2}%
+\rho_{j,j+1}^{2}-2\rho_{j-1,j}^{2}\rho_{j,j+1}^{2}}{1-\rho_{j-1,j}^{2}%
\rho_{j,j+1}^{2}}\right)  ^{k/2}\label{CH3}\\
\times H_{k}\left(  \frac{y\rho_{j-1,j}\left(  1-\rho_{j,j+1}^{2}\right)
+z\rho_{j,j+1}\left(  1-\rho_{j-1,j}^{2}\right)  }{\sqrt{\left(
1-\rho_{j-1,j}^{2}\rho_{j,j+1}^{2}\right)  \left(  \rho_{j-1,j}^{2}%
+\rho_{j,j+1}^{2}-2\rho_{j-1,j}^{2}\rho_{j,j+1}^{2}\right)  }}\right)
.\nonumber
\end{gather}

\end{theorem}

We now restrict our attention to $2-$ and $3-$dimensional normal
distributions. By studying these cases, we will learn how to use Theorem
\ref{COND} to calculate conditional moments $h_{k_{1}\ldots,k_{n-1}%
}(\mathbf{\Sigma,}z\mathbf{)\allowbreak}$ iteratively.

\section{Particular cases $2$ and $3$ and some general observations\label{23}}

We start with the case $n\allowbreak=\allowbreak2.$ Let us assume that the
random vector $(X,Y),$ has $N\left(  \mathbf{0,}\left[
\begin{array}
[c]{cc}%
1 & \rho\\
\rho & 1
\end{array}
\right]  \right)  $ distribution . It is common knowledge that then%
\[
E(H_{n}(X)|Y=y)=\rho^{n}H_{n}(y).
\]
Note that the formula (\ref{E|z}) for $n\allowbreak=\allowbreak2$ reads%

\[
\sum_{s=0}^{n}\binom{n}{s}\hat{C}_{s}\rho^{n-s}H_{n-s}(z)=h_{n}(z),
\]
where $\hat{C}_{n}\allowbreak=\allowbreak EH_{n}(X)$, with $X\sim f_{N}(x).$
Clearly, $\hat{C}_{n}\allowbreak=\allowbreak\delta_{0,n}$ . Hence,
$h_{n}(z)\allowbreak=\allowbreak\rho^{n}H_{n}(z).$

Consequently, we have%
\begin{equation}
C_{n,m}(\rho)=EH_{n}(X)H_{m}(Y)=%
\begin{cases}
0\text{,} & \text{if }n\neq m\text{;}\\
n!\rho^{n}\text{,} & \text{if }n=m\text{.}%
\end{cases}
\label{C2}%
\end{equation}
Now notice that the coefficient (\ref{coef}) of the expansion (\ref{rozkl}) is
equal
\[
c_{n,m}=%
\begin{cases}
0\text{,} & \text{if }n\neq m\text{;}\\
\rho^{n}/n!\text{,} & \text{if }n=m\text{.}%
\end{cases}
\]
Thus, for all $x,y\in\mathbb{R},$ we have
\begin{align*}
&  \frac{1}{2\pi\sqrt{1-\rho^{2}}}\exp\left(  -\frac{x^{2}-2\rho xy+y^{2}%
}{2(1-\rho^{2}))}\right) \\
&  =f_{n}(x)f_{n}(y)\sum_{n\geq0}\frac{\rho^{n}}{n!}H_{n}(x)H_{n}(y).
\end{align*}
Thus, we obtain the well-known Mehler formula.

\begin{theorem}
Let $\left(  X,Y,Z\right)  \sim N\left(
\begin{array}
[c]{c}%
0\\
0\\
0
\end{array}
,\left[
\begin{array}
[c]{ccc}%
1 & \rho_{12} & \rho_{13}\\
\rho_{12} & 1 & \rho_{23}\\
\rho_{13} & \rho_{23} & 1
\end{array}
\right]  \right)  ,$ then we have

1. conditional distribution $X|Y,Z$ is also normal%
\[
N\left(  y(\frac{\rho_{12}-\rho_{13}\rho_{23}}{1-\rho_{23}^{2}})+z\left(
\frac{\rho_{13}-\rho_{12}\rho_{23}}{1-\rho_{23}^{2}}\right)  ,\frac
{1-\rho_{12}^{2}-\rho_{13}^{2}-\rho_{23}^{2}+2\rho_{12}\rho_{13}\rho_{23}%
}{1-\rho_{23}^{2}}\right)  ,
\]

2. conditional distribution $%
\genfrac{[}{]}{0pt}{}{X}{Y}%
|Z$ is two-dimensional normal
\[
N\left(
\genfrac{[}{]}{0pt}{}{\rho_{13}z}{\rho_{23}z}%
,\left[
\begin{array}
[c]{cc}%
1-\rho_{13}^{2} & \rho_{12}-\rho_{13}\rho_{23}\\
\rho_{12}-\rho_{13}\rho_{23} & 1-\rho_{23}^{2}%
\end{array}
\right]  \right)  .a.s..
\]
In particular $\ $we have%
\[
E(H_{n}(Y)|Z=z)=\rho_{23}^{2}H_{n}(z).
\]

3. Almost surely we have
\begin{gather*}
E(H_{n}(X)|Y=y,Z=z)\\
=\frac{n!}{\left(  1-\rho_{23}^{2}\right)  ^{n}}\sum_{s=0}^{\left\lfloor
n/2\right\rfloor }\frac{(-1)^{s}\rho_{23}^{s}(\rho_{12}-\rho_{13}\rho
_{23})^{s}(\rho_{13}-\rho_{12}\rho_{23})^{s}}{s!(n-2s)!}\\
\times\sum_{k=0}^{n-2s}\binom{n-2s}{k}\left(  \rho_{12}-\rho_{13}\rho
_{23}\right)  ^{k}\left(  \rho_{13}-\rho_{12}\rho_{23}\right)  ^{n-2s-k}%
H_{k}(y)H_{n-2s-k}(z).
\end{gather*}

4. Almost surely we have
\begin{equation}
E(H_{n}(X)H_{m}(Y)|Z=z)=\sum_{j=0}^{\min(n,m)}\binom{n}{j}\binom{m}{j}%
j!\rho_{12}^{j}\rho_{13}^{n-j}\rho_{23}^{m-j}H_{n+m-2j}(z). \label{EXYz}%
\end{equation}

\end{theorem}

\begin{proof}
To get 1. we use (\ref{margin}) with $\mathbf{\mu}_{1}\allowbreak
=\allowbreak0,\mathbf{\mu}_{2}\allowbreak=\allowbreak0,$ $\mathbf{\Sigma}%
_{11}\allowbreak=\allowbreak1,$ $\mathbf{\Sigma}_{12}\allowbreak
=\allowbreak\left[
\begin{array}
[c]{cc}%
\rho_{12} & \rho_{13}%
\end{array}
\right]  ,$ $\mathbf{\Sigma}_{22}\allowbreak=\allowbreak\left[
\begin{array}
[c]{cc}%
1 & \rho_{23}\\
\rho_{23} & 1
\end{array}
\right]  $. To get 2. we use also (\ref{margin}), this time with that
$\mathbf{\mu}_{1}\allowbreak=\allowbreak0,\mathbf{\mu}_{2}\allowbreak
=\allowbreak0,$ $\mathbf{\Sigma}_{11}\allowbreak=\allowbreak\left[
\begin{array}
[c]{cc}%
1 & \rho_{12}\\
\rho_{12} & 1
\end{array}
\right]  1,$ $\mathbf{\Sigma}_{12}\allowbreak=\allowbreak\left[
\begin{array}
[c]{c}%
\rho_{13}\\
\rho_{23}%
\end{array}
\right]  ,$ $\mathbf{\Sigma}_{22}\allowbreak=1.$

To get 3 we use results of the auxiliary Lemma \ref{Her} with%
\[
m\allowbreak=\allowbreak y(\frac{\rho_{12}-\rho_{13}\rho_{23}}{1-\rho_{23}%
^{2}})+z\left(  \frac{\rho_{13}-\rho_{12}\rho_{23}}{1-\rho_{23}^{2}}\right)
\]
and
\[
\sigma^{2}\allowbreak=\allowbreak\frac{1-\rho_{12}^{2}-\rho_{13}^{2}-\rho
_{23}^{2}+2\rho_{12}\rho_{13}\rho_{23}}{1-\rho_{23}}.
\]
Let us select
\begin{align*}
a\allowbreak &  =\allowbreak\sqrt{\left(  (\frac{\rho_{12}-\rho_{13}\rho_{23}%
}{1-\rho_{23}^{2}}\right)  ^{2}+\left(  \frac{\rho_{13}-\rho_{12}\rho_{23}%
}{1-\rho_{23}^{2}}\right)  ^{2}}\allowbreak\\
&  =\allowbreak\frac{\sqrt{(\rho_{13}^{2}+\rho_{12}^{2})(1+\rho_{23}%
^{2})-4\rho_{12}\rho_{13}\rho_{23}}}{\left(  1-\rho_{23}^{2}\right)  }%
\end{align*}
and consequently
\[
\sigma^{2}-1+a^{2}\allowbreak=\allowbreak\frac{-2\rho_{23}(\rho_{12}-\rho
_{13}\rho_{23})(\rho_{13}-\rho_{12}\rho_{23})}{(1-\rho_{23}^{2})^{2}}.
\]

We get further
\begin{gather*}
E(H_{n}(X)|Y=y,Z=z)\\
=\sum_{s=0}^{\left\lfloor n/2\right\rfloor }\frac{n!a^{n-2s}}{2^{s}%
(n-2s)!s!}H_{n-2s}(m/a)(\sigma^{2}+a^{2}-1)^{s}\\
=\sum_{s=0}^{\left\lfloor n/2\right\rfloor }\frac{n!a^{n-2s}(-1)^{s}\rho
_{23}^{s}(\rho_{12}-\rho_{13}\rho_{23})^{s}(\rho_{13}-\rho_{12}\rho_{23})^{s}%
}{(n-2s)!s!(1-\rho_{23}^{2})^{2s}}\times\\
H_{n-2s}\left(  \frac{\allowbreak y(\frac{\rho_{12}-\rho_{13}\rho_{23}}%
{1-\rho_{23}^{2}})+z\left(  \frac{\rho_{13}-\rho_{12}\rho_{23}}{1-\rho
_{23}^{2}}\right)  }{\sqrt{\left(  (\frac{\rho_{12}-\rho_{13}\rho_{23}}%
{1-\rho_{23}^{2}}\right)  ^{2}+\left(  \frac{\rho_{13}-\rho_{12}\rho_{23}%
}{1-\rho_{23}^{2}}\right)  ^{2}}}\right)  .
\end{gather*}
Now we apply umbral identity (\ref{H(sum)}) and get
\begin{gather*}
E(H_{n}(X)|Y=y,Z=z)\\
=\sum_{s=0}^{\left\lfloor n/2\right\rfloor }\frac{n!a^{n-2s}}{2^{s}%
s!(n-2s)!}H_{n-2s}(m/a)\left(  \sigma^{2}+a^{2}-1\right)  ^{s}\\
=\sum_{s=0}^{\left\lfloor n/2\right\rfloor }\frac{n!\left(  \sigma^{2}%
+a^{2}-1\right)  ^{s}}{2^{s}s!(n-2s)!}\\
\times\sum_{k=0}^{n-2s}\binom{n-2s}{k}\left(  \frac{\rho_{12}-\rho_{13}%
\rho_{23}}{1-\rho_{23}^{2}}\right)  ^{k}\left(  \frac{\rho_{13}-\rho_{12}%
\rho_{23}}{1-\rho_{23}^{2}}\right)  ^{n-2s-k}H_{k}(y)H_{n-2s-k}(z)\\
=\frac{n!}{\left(  1-\rho_{23}^{2}\right)  ^{n}}\sum_{s=0}^{\left\lfloor
n/2\right\rfloor }\frac{(-1)^{s}\rho_{23}^{s}(\rho_{12}-\rho_{13}\rho
_{23})^{s}(\rho_{13}-\rho_{12}\rho_{23})^{s}}{s!(n-2s)!}\\
\times\sum_{k=0}^{n-2s}\binom{n-2s}{k}\left(  \rho_{12}-\rho_{13}\rho
_{23}\right)  ^{k}\left(  \rho_{13}-\rho_{12}\rho_{23}\right)  ^{n-2s-k}%
H_{k}(y)H_{n-2s-k}(z).
\end{gather*}

4. We have $\hat{C}_{k_{1},k_{2}}\allowbreak=\allowbreak C_{k_{1},k_{2}%
m}^{(2)}(\mathbf{D\Sigma}_{c}\mathbf{D})\prod_{j=1}^{2}\left(  \sqrt
{1-\rho_{j,3}^{2}}\right)  ^{k_{j}}\allowbreak=\allowbreak$\newline$%
\begin{cases}
0\text{,} & \text{if }k_{1}\neq k_{2}\text{;}\\
k_{1}!(\rho_{12}-\rho_{13}\rho_{14})^{k_{1}}\text{,} & \text{if }k_{1}%
=k_{2}\text{.}%
\end{cases}
$ This is so since
\[
\left(  \frac{\rho_{12}-\rho_{13}\rho_{14}}{\sqrt{(1-\rho_{13}^{2}%
)(1-\rho_{14}^{3})}}\right)  ^{k_{1}}\left(  1-\rho_{13}^{2}\right)
^{k_{1}/2}\left(  1-\rho_{14}^{2}\right)  ^{k_{1}/2}=(\rho_{12}-\rho_{13}%
\rho_{14})^{k_{1}}.
\]
Hence, following (\ref{E|z}) we get%
\begin{gather*}
h_{n,m}(z)=\sum_{j=0}^{n}\sum_{k=0}^{m}\binom{n}{j}\binom{m}{k}\hat{C}%
_{j,k}\rho_{13}^{n-j}\rho_{14}^{m-k}H_{n-j}(z)H_{m-k}(z)\\
=\sum_{j=0}^{\min(n,m)}\binom{n}{j}\binom{m}{j}j!(\rho_{12}-\rho_{13}\rho
_{14})^{j}\rho_{13}^{n-j}\rho_{14}^{m-j}H_{n-j}(z)H_{m-j}(z)\\
=\sum_{j=0}^{\min(n,m)}\binom{n}{j}\binom{m}{j}j!(\rho_{12}-\rho_{13}\rho
_{14})^{j}\rho_{13}^{n-j}\rho_{14}^{m-j}\\
\times\sum_{k=0}^{\min(n-j,m-j)}\binom{n-j}{k}\binom{m-j}{k}k!H_{n+m-2j-2k}%
(z)\\
=\sum_{s=0}^{\min(n,m)}\binom{n}{s}\binom{m}{s}s!\rho_{13}^{n-s}\rho
_{14}^{m-s}H_{n+m-2s}(z)\sum_{j=0}^{s}\binom{s}{j}(\rho_{12}-\rho_{13}%
\rho_{14})^{j}\left(  \rho_{13}\rho_{14}\right)  ^{s-j}\\
=\sum_{s=0}^{\min(n,m)}\binom{n}{s}\binom{m}{s}s!\rho_{12}^{s}\rho_{13}%
^{n-s}\rho_{14}^{m-s}H_{n+m-2s}(z).
\end{gather*}

\end{proof}

Having calculated $E(H_{n}(X)H_{m}(Y)|Z=z)$ we can easily calculate moments
\[
EH_{n}(X)H_{m}(Y)H_{n}(Z)=E\left(  H_{k}(Z)E\left(  H_{n}(X)H_{m}(Y)|Z\right)
\right)  ,
\]
of $N\left(
\begin{array}
[c]{c}%
0\\
0\\
0
\end{array}
,\left[
\begin{array}
[c]{ccc}%
1 & \rho_{12} & \rho_{13}\\
\rho_{12} & 1 & \rho_{23}\\
\rho_{13} & \rho_{23} & 1
\end{array}
\right]  \right)  $ distribution. We see that
\[
EH_{n}(X)H_{m}(Y)H_{n}(Z)=\sum_{j=0}^{\min(n,m)}\binom{n}{j}\binom{m}{j}%
j!\rho_{12}^{j}\rho_{13}^{n-j}\rho_{23}^{m-j}E\left(  H_{n+m-2j}%
(Z)H_{k}(Z)\right)  .
\]
Consequently, we get%
\[
EH_{n}(X)H_{m}(Y)H_{n}(Z)=\frac{m!n!k!}{(\frac{m+k-n}{2})!\left(  \frac
{n+m-k}{2}\right)  !\left(  \frac{n+k-m}{2}\right)  !}\rho_{12}^{\frac
{n+m-k}{2}}\rho_{13}^{\frac{n+k-m}{2}}\rho_{23}^{\frac{m+k-n}{2}}.
\]
Hence, the coefficient $c_{n,m,k}$ in the expansion (\ref{rozkl}) are equal to:%
\begin{equation}
c_{n,m,k}=\frac{1}{n!m!k!}EH_{n}(X)H_{m}(Y)H_{n}(Z)=\frac{\rho_{12}%
^{\frac{n+m-k}{2}}\rho_{13}^{\frac{n+k-m}{2}}\rho_{23}^{\frac{m+k-n}{2}}%
}{(\frac{m+k-n}{2})!\left(  \frac{n+m-k}{2}\right)  !\left(  \frac{n+k-m}%
{2}\right)  !}. \label{3H}%
\end{equation}
Let us denote for brevity $\mathbf{t}^{T}\allowbreak=\allowbreak(x,y,z),$
$\mathbf{\Sigma\allowbreak=\allowbreak}\left[
\begin{array}
[c]{ccc}%
1 & \rho_{12} & \rho_{13}\\
\rho_{12} & 1 & \rho_{23}\\
\rho_{13} & \rho_{23} & 1
\end{array}
\right]  .$ We have for all $x,y,z\in\mathbb{R}$
\begin{gather*}
\frac{1}{\sqrt{(2\pi)^{3}\det\mathbf{\Sigma}}}\exp\left(  -\frac
{\mathbf{t\Sigma}^{-1}\mathbf{t}}{2}\right)  =f_{N}(x)f_{N}\left(  y\right)
f_{N}(z)\\
\times\sum_{\substack{n,m,k\geq0 \\n+m+k\text{ is even}}}\frac{\rho
_{12}^{\frac{n+m-k}{2}}\rho_{13}^{\frac{n+k-m}{2}}\rho_{23}^{\frac{m+k-n}{2}}%
}{(\frac{m+k-n}{2})!\left(  \frac{n+m-k}{2}\right)  !\left(  \frac{n+k-m}%
{2}\right)  !}H_{n}(x)H_{m}(y)H_{k}(z).
\end{gather*}
Now notice that when denoting $p$\allowbreak=\allowbreak$(n+m-k)/2,$
$q\allowbreak=\allowbreak\left(  n+k-m\right)  /2,$ $r\allowbreak
=\allowbreak\left(  m+k-n\right)  /2,$ we get $p+q\allowbreak=\allowbreak n,$
$p+r\allowbreak=\allowbreak m,$ $q+r\allowbreak=\allowbreak k$ and we can
change the order of summation in the last sum getting%
\begin{align*}
&  \frac{1}{\sqrt{(2\pi)^{3}\det\mathbf{\Sigma}}}\exp\left(  -\frac
{\mathbf{t\Sigma}^{-1}\mathbf{t}}{2}\right) \\
&  =f_{N}(x)f_{N}\left(  y\right)  f_{N}(z)\sum_{p,q,r\geq0}\frac{\rho
_{12}^{p}\rho_{13}^{q}\rho_{23}^{r}}{p!q!r!}H_{p+q}(x)H_{p+r}(y)H_{q+r}(z),
\end{align*}
that is exactly the three-dimensional version of the Kibble-Slepian formula.

\begin{remark}
[Case $n=4$]\label{n-4}Let us calculate the quantity \newline$C_{n,m,k}%
(\mathbf{D\Sigma}_{c}\mathbf{D})\left(  \sqrt{1-\rho_{14}^{2}}\right)
^{n}\left(  \sqrt{1-\rho_{2,4}^{2}}\right)  ^{m}\left(  \sqrt{1-\rho_{3,4}%
^{2}}\right)  ^{k},$ with $\mathbf{\Sigma}_{c}\allowbreak=\allowbreak
\mathbf{\Sigma-\rho\rho}^{T},$ where $\mathbf{\rho}^{T}\allowbreak
=\allowbreak(\rho_{14},\rho_{24,}\rho_{34}),$ and $\mathbf{D}$ is the diagonal
matrix with $1/\sqrt{1-\rho_{i4}^{2}}$ on $i$ position of diagonal ,
$i\allowbreak=\allowbreak1,2,3$. Recall that this quantity is needed in order
to use equation (\ref{CMom}), for $n\allowbreak=\allowbreak4$. We have for
$n+m+k$ being even and $m+m\geq k$ and $n+k\geq m$ and $m+k\geq n$
\begin{gather}
\hat{C}_{n,m,k}=C_{n,m,k}(\mathbf{D\Sigma}_{c}\mathbf{D})\left(  \sqrt
{1-\rho_{14}^{2}}\right)  ^{n}\left(  \sqrt{1-\rho_{2,4}^{2}}\right)
^{m}\left(  \sqrt{1-\rho_{3,4}^{2}}\right)  ^{k}\nonumber\\
=\left(  \sqrt{1-\rho_{14}^{2}}\right)  ^{n}\left(  \sqrt{1-\rho_{2,4}^{2}%
}\right)  ^{m}\left(  \sqrt{1-\rho_{3,4}^{2}}\right)  ^{k}\frac{m!n!k!}%
{(\frac{m+k-n}{2})!\left(  \frac{n+m-k}{2}\right)  !\left(  \frac{n+k-m}%
{2}\right)  !}\nonumber\\
\times\left(  \frac{\rho_{12}-\rho_{14}\rho_{24}}{\sqrt{\left(  1-\rho
_{14}^{2}\right)  \left(  1-\rho_{24}^{2}\right)  }}\right)  ^{\frac{n+m-k}%
{2}}\left(  \frac{\rho_{13}-\rho_{14}\rho_{34}}{\sqrt{\left(  1-\rho_{14}%
^{2}\right)  \left(  1-\rho_{34}^{2}\right)  }}\right)  ^{\frac{n+k-m}{2}%
}\nonumber\\
\times\left(  \frac{\rho_{23}-\rho_{24}\rho_{34}}{\sqrt{\left(  1-\rho
_{24}^{2}\right)  \left(  1-\rho_{34}^{2}\right)  }}\right)  ^{\frac{m+k-n}%
{2}}\nonumber\\
=\frac{m!n!k!\left(  \rho_{12}-\rho_{14}\rho_{24}\right)  ^{\frac{n+m-k}{2}%
}\left(  \rho_{13}-\rho_{14}\rho_{34}\right)  ^{\frac{n+k-m}{2}}\left(
\rho_{23}-\rho_{24}\rho_{34}\right)  ^{\frac{m+k-n}{2}}}{(\frac{m+k-n}%
{2})!\left(  \frac{n+m-k}{2}\right)  !\left(  \frac{n+k-m}{2}\right)  !}
\label{pom2}%
\end{gather}

From (\ref{E|z}) we deduce that
\begin{align*}
h_{n,m,k}(\mathbf{\Sigma,}z)  &  =\sum_{s_{1}=0}^{n}\sum_{s_{2}=0}^{m}%
\sum_{s_{3}=0}^{k}\hat{C}_{s_{1},s_{2},s_{3}}\binom{n}{s_{1}}\binom{m}{s_{2}%
}\binom{k}{s_{3}}\\
&  \times\rho_{14}^{n-s_{1}}\rho_{24}^{m-s_{2}}\rho_{34}^{k-s_{3}}H_{n-s_{1}%
}(z)H_{m-s_{2}}(z)H_{k-s_{3}}(z),
\end{align*}
where we denoted (\ref{pom2}) by $\hat{C}_{n,m,k}\mathbf{\ }$for simplicity.
We can expect polynomials $h_{n,m,k}(\mathbf{\Sigma,}z)$ to be of the form%
\[
h_{n,m,k}(\mathbf{\Sigma,}z)=\sum_{j=0}^{n+m+k}\beta_{n+m+k-2j}(\mathbf{\Sigma
)}H_{n+m+k-2j}(z),
\]
where coefficients $\beta_{n+m+k-2j}(\mathbf{\Sigma)}$ are certain polynomials
in $6$ off diagonal elements of the matrix $\mathbf{\Sigma}$. These
polynomials are nontrivial, rather difficult to guess, as we know from the
$4-$dimensional form of the Kibble-Slepian formula.
\end{remark}

\begin{example}
With the help of Mathematica we can calculate for example
\begin{align*}
h_{1,1,1}(\mathbf{\Sigma},z)  &  =\rho_{14}\rho_{24}\rho_{34}H_{3}%
(z)+(\rho_{14}\rho_{23}+\rho_{13}\rho_{24}+\rho_{12}\rho_{34})H_{1}(z),\\
h_{2,1,1}(\mathbf{\Sigma},z)  &  =\rho_{14}^{2}\rho_{24}\rho_{34}H_{4}%
(z)+\rho_{14}(\rho_{14}\rho_{23}+2\rho_{12}\rho_{34}+2\rho_{13}\rho_{24}%
)H_{2}(z)+2\rho_{12}\rho_{13.}%
\end{align*}

\end{example}

\end{document}